\documentclass{amsproc}
\usepackage{amsmath}
\usepackage{enumerate}
\usepackage{amsmath,amsthm,amscd,amssymb}
\usepackage{latexsym}
\usepackage{upref}
\usepackage{verbatim}

\usepackage[mathscr]{eucal}
\usepackage{dsfont}

\usepackage{graphicx}
\usepackage[colorlinks,hyperindex,hypertex]{hyperref}

\usepackage{hhline}

\usepackage[OT2,OT1]{fontenc}
\newcommand\cyr
{
\renewcommand\rmdefault{wncyr}
\renewcommand\sfdefault{wncyss}
\renewcommand\encodingdefault{OT2}
\normalfont
\selectfont
}
\DeclareTextFontCommand{\textcyr}{\cyr}

\newtheorem{theorem}{Theorem}%[section]

\newtheorem{definition}[theorem]{Definition}
\newtheorem{remark}[theorem]{Remark}
\newtheorem{hypothesis}[theorem]{Hypothesis}

\chardef\bslash=`\\ % p. 424, TeXbook

\newcommand{\wh}{\widehat}

\newcommand{\dA}{{\dot A}}

\newcommand{\bbR}{{\mathbb{R}}}

\newcommand{\bbC}{{\mathbb{C}}}

\newcommand{\ti}{\tilde  }

\newcommand{\dom}{\text{\rm{Dom}}}

\newcommand{\calD}{{\mathcal D}}

\newcommand{\calH}{{\mathcal H}}

\newcommand{\calN}{{\mathcal N}}

\newcommand{\calR}{{\mathcal R}}

\newcommand{\calS}{{\mathcal S}}

\newcommand{\mM}{\mathfrak M}

\newcommand{\whA}{T}

\renewcommand{\Im}{\text{\rm Im}}

\def\sM{{\mathfrak M}}   \def\sN{{\mathfrak N}}

\def\bA{{\mathbb A}}      \def\dC{{\mathbb C}}

      \def\dR{{\mathbb R}}

   \def\cH{{\mathcal H}}   
      
   \def\cN{{\mathcal N}}   
      \def\cR{{\mathcal R}}
\def\cS{{\mathcal S}}

\def\RE{{\rm Re\,}}
\def\Ker{{\rm Ker\,}}

\def\wh{\hat}

\def\uphar{{\upharpoonright\,}}

\DeclareMathOperator{\IM}{Im}

\newcommand{\eval}[2][\right]{\relax
  \ifx#1\right\relax \left.\fi#2#1\rvert}

\begin{document}

\title{The c-Entropy  optimality of Donoghue classes}

\author{S. Belyi}
\address{Department of Mathematics\\ Troy University\\
Troy, AL 36082, USA\\
%URL: {\sf http://spectrum.troy.edu/$\sim$belyi/}
}
\curraddr{}
\email{sbelyi@troy.edu}
%\thanks{Thank you.}


\author[K. A. Makarov]{K. A. Makarov}
\address{Department of Mathematics\\
 University of Missouri\\
  Columbia, MO 63211, USA}
\email{makarovk@missouri.edu}

%\thanks{The second author was partially supported by the Simons collaboration grant  00061759  while preparing this paper.}

%    Information for third author
\author{E. Tsekanovskii}
\address{Department of Mathematics, Niagara University,  Lewiston,
NY  14109, USA} \email{\tt tsekanov@niagara.edu}

%\thanks{Thank you.}

%    General info
\date{DD/MM/2004}

%\dedicatory{In loving memory of Moshe Liv\u sic}


\begin{abstract}
In this note we evaluate c-Entropy  of perturbed L-systems introduced in \cite{BMkT-3}. Explicit formulas relating the   c-Entropy of
the   L-systems and the perturbation parameter  are established.  We also show that  c-Entropy attains its maximum value (finite or infinite) whenever  the perturbation parameter vanishes so that the impedance function of such a L-system belongs to one of the generalized (or regular) Donoghue classes.
 \end{abstract}

\maketitle

\tableofcontents

%%%%%%%%%%%%%%%%%

%%%%%%%%%%%%%%%%%%%
\section{Introduction}\label{s1}

This paper is  {devoted} to the  study of the connections between various subclasses of Herglotz-Nevanlinna functions and their realizations as  the impedance functions of conservative  L-systems (see \cite{ABT,BMkT,BMkT-2,BMkT-3,BT-21,Lv2}).

Recall the concept of a conservative L-system.

Let $T$ be a non-symmetric, densely defined, closed,  dissipative  linear operator in a Hilbert space $\cH$.
%Let $T$ be a non-symmetric, densely defined, and closed linear operator in a Hilbert space $\cH$ such that its resolvent set $\rho(T)$ is not empty.
We also assume that the lineal
$$\dom (\dot A)=\dom(T)\cap \dom(T^*)$$ is dense in $\cH$
and that the restriction $\dot A=T|_{\dom(\dot A)}$  is a closed symmetric operator with deficiency indices $(1,1)$.

  Let $\calH_+\subset\calH\subset\calH_-$ be the rigged Hilbert space associated with the symmetric operator $\dot A$ (see the next section for details).

By  an  \textit{L-system} we mean the array
\begin{equation}
\label{col0}
 \Theta =
\left(%
\begin{array}{ccc}
  \bA    & K & 1 \\
   \calH_+\subset\calH\subset\calH_- &  & \dC \\
\end{array}%
\right),
\end{equation}
where the \textit{state-space operator} $\bA$ is a bounded linear operator from
$\calH_+$ into $\calH_-$ such that  $\dA \subset T\subset \bA$, $\dA \subset T^* \subset \bA^*$,
$K$ is a bounded linear operator from $\dC$ into $\calH_-$  such that
$\IM\bA=KK^*$.

%The operator $T$ above is the \textit{main operator} of the L-system $\Theta$ that defines $\Theta$ uniquely together with the \textit{quasi-kernel} $\hat A$ that is the self-adjoint extension of $\dA$ such that $\RE\bA\supset\hat A=\hat A^*\supset\dA$.
{In the framework of the approach in question the} operator-valued function
\begin{equation*}\label{W1}
 W_\Theta(z)=I-2iK^*(\bA-zI)^{-1}K,\quad z\in \rho(T),
\end{equation*}
 is called the \textit{transfer function}  of an L-system $\Theta$ and
\begin{equation*}\label{real2}
 V_\Theta(z)=i[W_\Theta(z)+I]^{-1}[W_\Theta(z)-I] =K^*(\RE\bA-zI)^{-1}K,\quad z\in\rho(T)\cap\dC_{\pm},
\end{equation*}
is {named}  the \textit{impedance function } of $\Theta$. The formal definition of  L-systems
{is} presented in Section \ref{s2}.

 From the analytic standpoint, the main role in our  considerations is played by  the generalized Donoghue classes introduced  and discussed in  \cite{BMkT}, \cite{BMkT-2},  \cite{BT-16}, \cite{BT-21}.
 Recall   that   the standard Donoghue class $\sM$ consists of  all analytic  analytic functions $M(z)$ that  admit the representation
\begin{equation}\label{murep}
M(z)=\int_\bbR \left
(\frac{1}{\lambda-z}-\frac{\lambda}{1+\lambda^2}\right )
d\mu(\lambda), \quad z\in \bbC_+,
\end{equation}
for some infinite Borel measure $\mu(d\lambda)$ such that
\begin{equation}\label{norm}
\int_\bbR
\frac{d\mu(\lambda)}{1+\lambda^2}=1
\end{equation}
 (see, e.g., \cite{MT-S}).  Given that, the {\it generalized} Donoghue classes accommodate the functions  from  $\sM$ composed with  the  action of the  ``$ax+b$ group",  the group of affine transformations of $\bbR$ preserving the orientation.  Namely, for  $a>0 $ and
$ Q\in \bbR$ introduce the class of analytic mapping from the upper half-plane into itself
\begin{equation}\label{e-4-NR}
\calN_{a,Q}=\{a M+Q, M\in \sM\}, \quad a>0, \quad Q\in \bbR.
\end{equation}
As it follows from \cite{BMkT} (also see  \cite{BMkT-2,BT-16,BT-21}), the  mappings  from  $\calN_{a,Q}$ can be realized as the impedance functions of L-systems of the form \eqref{col0}.
One easily notices as well   that  the generalized Donoghue classes $\sM_\kappa$ and $\sM^{-1}_\kappa$
discussed in   \cite{BMkT}, \cite{BMkT-2},  \cite{BT-16}, \cite{BT-21}
and also  the  classes   $\sM^Q$, $\sM^Q_\kappa$, $\sM^{-1,Q}_\kappa$
introduced in \cite{BMkT-3}  by two of the authors
coincide with  the  class $\calN_{a,Q}$ defined by \eqref{e-4-NR} for a certain choice  of $a$ and $Q$.  For instance, $$\sM_\kappa
=\calN_{\frac{1-\kappa}{1+\kappa}, 0}\quad \text{and}\quad
\sM_\kappa^Q
=\calN_{\frac{1-\kappa}{1+\kappa}, Q}.$$

 We refer to the publication list  above where  L-systems of the form \eqref{col0} for which the   impedance function   falls into a particular  generalized  Donoghue class {$\sM$, $\sM_\kappa$, or $\sM^{-1}_\kappa$
 } are described in detail.  We also refer to  \cite[Section 10]{BMkT-3} where the concept of a \textit{perturbed L-system} was introduced and the membership of the corresponding  impedance functions to the  perturbed classes $\sM^Q$, $\sM^Q_\kappa$, or $\sM^{-1,Q}_\kappa$ was established.
(Notice  that  in the framework of  the traditional theory of  self-adjoint extensions of symmetric operators
 the representation  theorems for the  functions from the standard Donoghue class $\sM$ are also discussed in  \cite{MT-S}.)

% In addition to that, the concept of a \textit{perturbed L-system} was introduced in \cite[Section 10]{BMkT-3}. A perturbed L-system is based on an a given L-system whose impedance function belongs to classes $\sM$, $\sM_\kappa$, or $\sM^{-1}_\kappa$ and shares the state-space and symmetric operator with it. Moreover, the impedance function of a perturbed L-system belongs to the corresponding perturbed class $\sM^Q$, $\sM^Q_\kappa$, or $\sM^{-1,Q}_\kappa$.

%In \cite{BT-16} and \cite{BT-21} the authors put forward the concept of \textit{c-Entropy} of an L-system.

 The main goal of this note is to  show that  the  c-Entropy  introduced in \cite{BT-16,BT-21} of the L-system
with the  impedance function from the classes  $\sM^Q$, $\sM^Q_\kappa$, or $\sM^{-1,Q}_\kappa$ (i) attains  a maximum  whenever the perturbation parameter $Q$ is zero and (ii) vanished as $|Q|\to \infty$. { Notice that if the perturbation parameter $Q=0$, the classes $\sM^Q$, $\sM^Q_\kappa$, or $\sM^{-1,Q}_\kappa$ coincide with their canonical ``unperturbed" counterparts $\sM$, $\sM_\kappa$, or $\sM^{-1}_\kappa$ which, taking into account the above, yields  the optimality of c-Entropy for the L-system with the  impedance function from the unperturbed classes  $\sM$, $\sM_\kappa$, or $\sM^{-1}_\kappa$.}

The paper is organized as follows.

Section \ref{s2} contains necessary information on the L-systems theory.

In Section \ref{s3} we remind the formal definition and describe basic properties of regular and generalized Donoghue classes.

Section \ref{s4} provides us with the detailed explanation of L-systems' perturbation concept. Here we also present the formulas for the von Neumann parameters of the main operator of a perturbed L-system.

%Sections \ref{s5} and \ref{s6} contain the main results of the paper.

In Section \ref{s5} we recall the definition of c-Entropy and  relate the c-Entropy of a perturbed L-system with the perturbation parameter.

In Section \ref{s6} we {recap the definition} of the dissipation coefficient  introduced in  \cite{BT-16,BT-21} {and study  its }  behavior   as a function of  the perturbation parameter $Q$ and the c-Entropy of the corresponding unperturbed L-system. We remark that in case $Q=0$, the obtained results generalize those {obtained }  in \cite{BT-21}.

The main results of Sections \ref{s5} and \ref{s6} are { mapped out in the summary } Table \ref{Table-1}.

We  conclude our note  with providing  examples illuminating  the  main results.

For convenience of the reader, an explicit construction of an L-system  with a given state-space operator is presented in  Appendix  \ref{A1}.

\section{Preliminaries}\label{s2}

For a pair of Hilbert spaces $\calH_1$, $\calH_2$ denote by
$[\calH_1,\calH_2]$ the set of all bounded linear operators from
$\calH_1$ to $\calH_2$.

Given   a closed, densely defined,
symmetric operator  $\dA$ in a Hilbert space $\calH$ with inner product
$(f,g),f,g\in\calH$, introduce  the rigged Hilbert space (see \cite{ABT,Ber})
%\label{107}
$\calH_+\subset\calH\subset\calH_- ,$ where $\calH_+ =\dom(\dA^*)$ is the Hilbert space equipped with the inner product
%\label{108}
\begin{equation}\label{108}
(f,g)_+ =(f,g)+(\dA^* f, \dA^*g),\;\;f,g \in \dom(\dA^*),
\end{equation}
and $\cH_-$ is its dual, the space of continuous linear functionals with respect to the corresponding  norm $\|\cdot \|_+$.

Denote by $\calR$  the \textit{\textrm{Riesz-Berezansky   operator}} $\calR$ (see  \cite{ABT}, \cite{Ber}) which maps $\mathcal H_-$ onto $\mathcal H_+$ such
 that   $(f,g)=(f,\calR g)_+$ ($\forall f\in\calH_+$, $g\in\calH_-$) and
 $\|\calR g\|_+=\| g\|_-$.
 Thus,
\begin{equation}\label{e3-4}
\aligned (f,g)_-=(f,\calR g)=(\calR f,g)=(\calR f,\calR g)_+,\qquad
(f,g\in \mathcal H_-),\\
(u,v)_+=(u,\calR^{-1} v)=(\calR^{-1} u,v)=(\calR^{-1} u,\calR^{-1}
v)_-,\qquad (u,v\in \mathcal H_+).
\endaligned
\end{equation}
 Note that identifying the space conjugate to $\calH_\pm$ with $\calH_\mp$, we get that if $\bA\in[\calH_+,\calH_-]$, then $\bA^*\in[\calH_+,\calH_-]$ as well.

We will be mostly interested in the following type of quasi-self-adjoint bi-extensions.

 \textit{In what follows we assume that $\dA$ has deficiency indices $(1,1)$.}

\begin{definition}[Definition 4.3.1 \cite{ABT},]\label{star_ext}

Suppose that  $T$ is a quasi-self-adjoint extension of $\dA$, that is,
$$
\dA\subset T\subset\dA^*.
$$
An operator $\bA\in[\calH_+,\calH_-]$ is called  the  \textit{($*$)-extension } of $T$
if  $$\dA \subset T\subset \bA
 \quad \text{and}\quad \dA \subset  T^*\subset  \bA^*$$
 and the  restriction $\widehat A$  of   $\RE\bA$ on
 \[
\dom(\widehat A)=\{f\in\cH_+:(\RE\bA) f\in\cH\},
\]
the quasi-kernel of  $\RE\bA$,
is a self-adjoint extension of $\dA$
\end{definition}

Recall that an operator $\bA\in[\calH_+,\calH_-]$ is said to be a \textit{self-adjoint
bi-extension} of a symmetric operator $\dA$ if $\bA=\bA^*$ and $\bA
\supset \dA$.
 For an operator  $\bA\in[\calH_+,\calH_-]$,  the  restriction  $\hat A$, $ \hat A=\bA\uphar\dom(\hat A)$   of  $\bA$   on
\[
\dom(\hat A)=\{f\in\cH_+:\bA f\in\cH\}
\]
 will be  called the  \textit{quasi-kernel} of  $\bA$ (see \cite[Section 2.1]{ABT}, \cite{TSh1}).
In this case, according to the von Neumann Theorem (see \cite[Theorem 1.3.1]{ABT}) the domain of $\wh A$, which is a self-adjoint extension of $\dA$,  can be represented as
\begin{equation}\label{DOMHAT}
\dom(\hat A)=\dom(\dA)\oplus(I+U)\sN_{i},
\end{equation}
where von Neumann's parameter $U$ is both a $(\cdot)$-isometric as well as $(+)$-isometric operator from $\sN_i$ into $\sN_{-i}$ , with $$\sN_{\pm i}=\Ker (\dA^*\mp i I)$$  the deficiency subspaces of $\dA$.

The description of all $(*)$-extensions via the Riesz-Berezansky   operator $\calR$ can be found in \cite[Section 4.3]{ABT}.

The following definition is a ``lite" version of the definition of L-system given for a scattering L-system with
 one-dimensional input-output space. It is tailored for the case when the symmetric operator of an L-system has deficiency indices $(1,1)$. (The general definition of  an L-system can be found in \cite[Definition 6.3.4]{ABT}.)

\begin{definition}\label{defs} Given  a symmetric operator $\dot A$  with deficiency indices $(1,1)$,  its quasi-self-adjoint dissipative extension $T$, and  the rigged Hilbert space
 $\calH_+\subset\calH\subset\calH_-$
associated with $\dot A$,
 an array
\begin{equation}\label{e6-3-2}
\Theta= \begin{pmatrix} \bA&K&\ 1\cr \calH_+ \subset \calH \subset
\calH_-& &\dC\cr \end{pmatrix}
\end{equation}
 is called an \textbf{{L-system}}   if
$\mathbb  A$ is a   ($\ast $)-extension of
of $ T$  with
$$\IM\bA= KK^*,$$
where $K\in [\dC,\calH_-]$ and $K^*\in [\calH_+,\dC].$
\end{definition}

For the dissipative operator in Definition \ref{defs} we reserve the notation
$T$ and will call it the  \textit{main  operator } of the system, while
 the operator $\bA$  will be said to be  \textit{the  state-space operator } of the system $\Theta$.
 The operator  $K$ will be
traditionally called the \textit{channel operator} of the system $\Theta$.

It is easy to see that the operator $\bA$ of the system  \eqref{e6-3-2}  can be chosen in such a way  that $$\IM\bA=(\cdot,\chi)\chi, \quad \text{for some}\quad\quad \chi\in\calH_-$$ and $$K c=c\cdot\chi,\quad c\in\dC.$$

  A system $\Theta$ in \eqref{e6-3-2} is called \textit{minimal} if the operator $\dA$ is a prime operator in $\calH$, i.e., there exists no non-trivial reducing invariant subspace of $\calH$ on which it induces a self-adjoint operator. Notice that minimal L-systems of the form \eqref{e6-3-2} with  one-dimensional input-output space were also discussed in \cite{BMkT}.

We  associate with an L-system $\Theta$ two  analytic functions,  the \textbf{transfer  function} of the L-system $\Theta$
\begin{equation}\label{e6-3-3}
W_\Theta (z)=I-2iK^\ast (\mathbb  A-zI)^{-1}K,\quad z\in \rho (T),
\end{equation}
and also the \textbf{impedance function}  given by the formula
\begin{equation}\label{e6-3-5}
V_\Theta (z) = K^\ast (\RE\bA - zI)^{-1} K, \quad z\in  \rho (\RE\bA),
\end{equation}

Recall that
 the  impedance function $V_\Theta(z)$ admits the  integral representation
\begin{equation}\label{hernev-real}
V_\Theta(z)=Q+\int_\bbR \left(\frac{1}{\lambda-z}-\frac{\lambda}{1+\lambda^2}\right)d\sigma,
\end{equation}
where $Q$ is a real number and $\sigma$ is an  infinite Borel measure   such that
$$
\int_\bbR\frac{d\sigma(\lambda)}{1+\lambda^2}<\infty.
$$

 The transfer function $W_\Theta (z)$ of the L-system $\Theta $ and function $V_\Theta (z)$ of the form (\ref{e6-3-5}) are connected by the following relations valid for $\IM z\ne0$, $z\in\rho(T)$,
\begin{equation}\label{e6-3-6}
\begin{aligned}
V_\Theta (z) &= i [W_\Theta (z) + I]^{-1} [W_\Theta (z) - I],\\
W_\Theta(z)&=(I+iV_\Theta(z))^{-1}(I-iV_\Theta(z)).
\end{aligned}
\end{equation}

In this context we refer to  \cite{ABT,BMkT,GT} and references therein for the description
of the class of all Herglotz-Nevanlinna functions that admit  realizations as impedance functions of an L-system.

\section{Donoghue classes and L-systems}\label{s3}

Denote by  $\calN$ (see \cite{BMkT-3}) the  class of all Herglotz-Nevanlinna functions $M(z)$ that admit the representation
\begin{equation}\label{hernev-0}
M(z)=\int_\bbR \left(\frac{1}{\lambda-z}-\frac{\lambda}{1+\lambda^2}\right)d\sigma,
\end{equation}
where $\sigma$ is an  infinite Borel measure.
%\textcolor{red}{Recall that in this case
$$
\int_\bbR\frac{d\sigma(\lambda)}{1+\lambda^2}<\infty.
$$
%}

 Following our earlier developments in \cite{BMkT,BMkT-3,MT10,MT2021}   denote by $\sM$,  $\sM_\kappa$ and  $\sM_\kappa^{-1}$ ($0\le\kappa<1$)  the subclass of $\calN$     with the  property
\begin{equation}\label{e-42-int-don}
\int_\bbR\frac{d\sigma(\lambda)}{1+\lambda^2}=1\,,\quad\text{equivalently,}\quad M(i)=i,
\end{equation}
\begin{equation}\label{e-38-kap}
\int_\bbR\frac{d\sigma(\lambda)}{1+\lambda^2}=\frac{1-\kappa}{1+\kappa}\,,\quad\text{equivalently,}\quad M(i)=i\,\frac{1-\kappa}{1+\kappa},
\end{equation}
and
\begin{equation}\label{e-39-kap}
\int_\bbR\frac{d\sigma(\lambda)}{1+\lambda^2}=\frac{1+\kappa}{1-\kappa}\,,\quad\text{equivalently,}\quad M(i)=i\,\frac{1+\kappa}{1-\kappa},
\end{equation}
respectively.

Clearly,  $$\sM=\sM_0=\sM_0^{-1}.$$

Recall that   \cite{D,GMT97,GT,MT-S} that $M\in \mM$ if and only if $M(z)$ can be realized  as the Weyl-Titchmarsh function $M_{(\dot A, A)}(z)$ associated with the pair $(\dot A, A)$ where
 $\dA$ is a closed prime densely defined symmetric operator with deficiency indices $(1,1)$,
 $A$ its self-adjoint extension and
\begin{equation}\label{e-DWT}
M_{(\dot A, A)}(z)=((Az+I)(A-zI)^{-1}g_+,g_+), \quad z\in \bbC_+,
\end{equation}
$$g_+\in \Ker( \dA^*-iI)\quad \text{with }\quad  \|g_+\|=1.$$

 If $M(z)$ is an arbitrary function from the class $\calN$ and the  normalization condition
\begin{equation}\label{e-66-L}
\int_\bbR\frac{d\sigma(\lambda)}{1+\lambda^2}=a
\end{equation}
holds
for some $a>0$, then it is easy to see that $M\in\sM$ if and only if $a=1$.  The membership of $M\in \cN$ in the other generalized Donoghue  classes
$ \sM_\kappa $ and $\sM_\kappa^{-1}$ can also be  easily described as follows:
 \begin{enumerate}
\item[] if $a<1$, then $M\in \sM_\kappa$ with
\begin{equation}\label{e-45-kappa-1}
\kappa=\frac{1-a}{1+a},
\end{equation}
\item[]and
\item[]if $a>1$, then $M\in \sM_\kappa^{-1}$ with
\begin{equation}\label{e-45-kappa-2}
\kappa=\frac{a-1}{1+a}.
 \end{equation}
 \end{enumerate}
 Throughout this Note we adopt the following hypothesis.

\begin{hypothesis}\label{setup} Suppose that  $\whA \ne\whA^*$  is a  maximal dissipative extension of a symmetric operator $\dot A$  with deficiency indices $(1,1)$.
Assume, in addition, that the deficiency elements $g_\pm\in \Ker (\dA^*\mp iI)$ are normalized, $\|g_\pm\|=1$, and chosen in such a way that
\begin{equation}\label{domT}
g_+-\kappa g_-\in \dom (\whA )\,\,\,\text{for some }
\,\,\, 0\le \kappa<1.
\end{equation}
Assume that  $A$ is a  self-adjoint extension of $\dot A$ such that  either
\begin{equation}\label{ddoomm14}
g_+- g_-\in \dom ( A)
\end{equation}
or
\begin{equation}\label{ddoomm14-1}
g_++ g_-\in \dom ( A).
\end{equation}
\end{hypothesis}

\begin{remark}\label{r-12}
If $T \ne T^*$ is a  maximal dissipative extension of $\dot A$,
$$
\Im(T f,f)\ge 0, \quad f\in \dom(T ),
$$
then $T$ is automatically quasi-self-adjoint  \cite{ABT, MT-S, MTBook} and therefore
\begin{equation}\label{parpar-1}
g_+-\kappa g_-\in \dom (T  )\quad \text{for some }
|\kappa|<1.
\end{equation}
In particular (see, e.g.,  \cite{MT-S}),
if $\kappa=0$, then  quasi-self-adjoint extension $\whA $ coincides with the restriction of the adjoint operator $\dot A^*$ on
$$
\dom(\whA )=\dom(\dot A)\dot + \Ker (\dA^*-iI).
$$

The requirement in \eqref{domT} that  $0\le \kappa<1$ does not really restricts the choice of  the main operator  $T$  of the systm (if $\kappa=|\kappa|e^{i\theta}$,
change (the basis) $g_-$ to $e^{i\theta}g_-$ in the deficiency subspace  $\Ker (\dA^*+ i I)$
to see that   \eqref{domT} is satisfied in the new basis, rather it imposes additional  requirements  (relative to $T$ ) on the self-adjoint reference operator $\widehat A$.

\end{remark}

\noindent
 As far as the generalized classes  $\sM_\kappa$ and  $\sM_\kappa^{-1}$, are concerned,
recall that  if the  main operator $T$ and the quasi-kernel $\hat A$ of $\RE\bA$
of an L-system $\Theta_1$ and $\Theta_2$  of the form \eqref{e6-3-2}   satisfy  Hypothesis \ref{setup} (\eqref{ddoomm14}  and  \eqref{ddoomm14-1}), respectively, then
 the impedance functions $V_{\Theta_1}(z)$ and $V_{\Theta_2}(z)$
 belong to the classes $\sM_\kappa$ and  $\sM_\kappa^{-1}$, respectively (see \cite{BMkT-2}).

\section{Perturbations of Donoghue classes and {the related} L-systems}\label{s4}

In this section we recall the definition of
``perturbed" versions $\sM^Q$, $\sM^Q_\kappa$, and $\sM^{-1,Q}_\kappa$
of the generalized Donoghue classes $\sM$, $\sM_\kappa$, and $\sM^{-1}_\kappa$ discussed in Section \ref{s3}
and briefly revisit the concept of a  ``perturbed" L-system introduced in \cite{BMkT-3}.

Given $Q\in \bbR\setminus\{0\}$, we say that  $V(z)\in\sM^Q$ if $V(z)$ admits the representation
 \begin{equation}\label{e-52-M-q}
V(z)= Q+\int_\bbR\left (\frac{1}{\lambda-z}-\frac{\lambda}{1+\lambda^2}\right )d\mu,\end{equation}
with
$$
 \int_\bbR\frac{d\mu(\lambda)}{1+\lambda^2}=1.
$$

If along with \eqref{e-52-M-q} the
normalization conditions \eqref{e-38-kap}, \eqref{e-39-kap} hold, we say that $V(z)$ belongs to the class
$\sM^Q_{\kappa}$,  $\sM^{-1,Q}_{\kappa}$, respectively.

%%%%%%%%%%%
\begin{figure}
  \begin{center}
  \includegraphics[width=90mm]{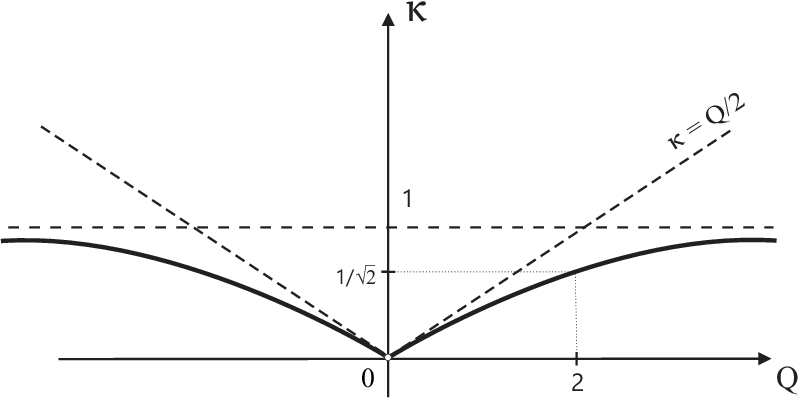}
  \caption{Class $\sM^Q$: Parameter $\kappa$ as a function of $Q$}\label{fig-1}
  \end{center}
\end{figure}
%%%%%%%%%%%

The following was shown in \cite[Theorem 10.1]{BMkT-3}. Let $\Theta_0$ be  an L-system   of the form \eqref{e6-3-2} satisfying the conditions of Hypothesis \ref{setup} \eqref{ddoomm14} and such that its impedance function $V_{\Theta_0}(z)$  belongs to the  class $\sM$.  Then for any real number $Q\ne0$ there exists another L-system $\Theta(Q)$  with the same symmetric operator $\dA$ as in $\Theta_0$ and such that
\begin{equation}\label{impshift1}
V_{\Theta(Q)}(z)=Q+V_{\Theta_0}(z)
\end{equation}
belongs to the class $\sM^Q$. 
In this case,  the von Neumann parameter  $\kappa(Q)$ of   its main operator $T(Q)$  is determined by  \begin{equation}\label{e-53-kappa'}
    \kappa(Q)=\frac{|Q|}{\sqrt{Q^2+4}},\quad Q\ne0.
 \end{equation}
  while the quasi-kernel $\hat A(Q)$ of $\RE\bA(Q)$ of the L-system $\Theta(Q)$ is defined by \eqref{DOMHAT} with \begin{equation}\label{e-54-U-M-q}
 U(Q)=\frac{Q}{|Q|}\cdot\frac{-Q+2i}{\sqrt{Q^2+4}},\quad Q\ne0.
\end{equation}
For the graph of $\kappa$ as a function of $Q$ see Figure \ref{fig-1}. We note that $\kappa(Q)$ is an even function whose derivative for $Q>0$ is
$$
\kappa'(Q)=\frac{4}{(Q^2+4)^{3/2}},\quad Q>0,
$$
giving the slope of the graph at $Q=0$ as $\kappa'(0+)=1/2$. The graph of the function is symmetric with respect to the $\kappa$-axis.

A similar  result (see \cite[Theorem 10.2]{BMkT-3}) takes place for the class $\sM_{\kappa}^Q$:
Let $\Theta_{\kappa}$ be  an L-system   of the form \eqref{e6-3-2} such that its impedance function $V_{\Theta_\kappa}(z)$  belongs to the  class $\sM_{\kappa}$.  Then for any real number $Q\ne0$ there exists another L-system $\Theta_\kappa(Q)$  with the same symmetric operator $\dA$ as in the system  $\Theta_{\kappa}$ and such that its impedance  function is obtained from
$V_{\Theta_{\kappa}}(z)$ by shifting by the constant $Q$, that is,
\begin{equation}\label{impshift2}
V_{\Theta_{\kappa}(Q)}(z)=Q+V_{\Theta_{\kappa}}(z).
\end{equation}
Notice that  $V_{\Theta_{\kappa}(Q)}\in  \sM_{\kappa}^Q$.

%%%%%%%%%%%
\begin{figure}
  \begin{center}
  \includegraphics[width=90mm]{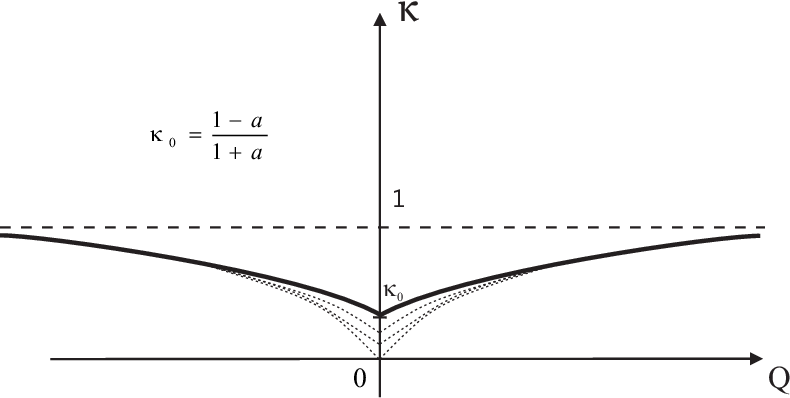}
  \caption{Class $\sM^Q_\kappa$  $(0<a<1)$: Parameter $\kappa$ as a function of $Q$}\label{fig-2}
  \end{center}
\end{figure}
%%%%%%%%%%%
 In this case, the von Neumann parameter  $\kappa(Q)$ of   the   main operator $T(Q)$  of the system $\Theta_\kappa(Q)$ is determined by the formula
  \begin{equation}\label{e-53-kappa-prime}
    \kappa(Q)=\frac{\left(b-2Q^2-\sqrt{b^2+4Q^2}\right)^2-a\left(b-\sqrt{b^2+4Q^2}\right)^2+4Q^2a(a-1)}{\left(b-2Q^2-\sqrt{b^2+4Q^2}\right)^2+a\left(b-\sqrt{b^2+4Q^2}\right)^2+4Q^2a(a+1)}.
 \end{equation}
Here
 \begin{equation}\label{e-78-b}
    b=Q^2+a^2-1
\end{equation}
with
$$
a=\frac{1-\kappa}{1+\kappa},
$$
while the quasi-kernel $\hat A(Q)$ of $\RE\bA(Q)$ of the L-system $\Theta_\kappa(Q)$ is defined by \eqref{DOMHAT}
%\begin{equation}\label{DOMHAT}
%\dom(\hat A)=\dom(\dA)\oplus(I+U)\sN_{i},
%\end{equation}
with
\begin{equation}\label{e-75-U}
    U(Q)=\frac{(a+Qi)(1-\kappa^2(Q))-1-\kappa^2(Q)}{2\kappa(Q)},\quad Q\ne0.
\end{equation}
%Here the von Neumann's parameter $U$ is both a $(\cdot)$-isometric as well as $(+)$-isometric operator from $\sN_i$ into $\sN_{-i}$ and $$\sN_{\pm i}=\Ker (\dA^*\mp i I)$$ are the deficiency subspaces of $\dA$.

The  graph of $\kappa$ as a function of $Q$ for this case is shown on the Figure \ref{fig-2}. Note that the vertex of the graph is located at the value of $$\kappa=\kappa_0=\frac{1-a}{1+a}.$$ Moreover, if $a\rightarrow 1^-$, then $\kappa_0\rightarrow 0$ as indicated by the dashed lines on the picture.
%\newpage

 Finally, (see \cite[Theorem 10.2]{BMkT-3}), for any L-system  $\Theta_{\kappa}$  of the form \eqref{e6-3-2} with
 $V_{\Theta_\kappa}(z)\in\sM_{\kappa}^{-1}$ and   any real number $Q\ne0$ there exists another L-system $\Theta_\kappa(Q)$  with the same symmetric operator $\dA$
as in $\Theta_{\kappa}$ and such that
\begin{equation}\label{impshift3}
V_{\Theta_{\kappa}(Q)}(z)=Q+V_{\Theta_{\kappa}}(z).
\end{equation}
In this case,  the von Neumann parameter  $\kappa(Q)$ of   its main operator $T(Q)$  is determined for $Q\ne0 $ by the formula
 \begin{equation}\label{e-85-kappa-prime}
    \kappa(Q)=\frac{a\left(b+\sqrt{b^2+4Q^2}\right)^2-\left(b-2Q^2+\sqrt{b^2+4Q^2}\right)^2-4Q^2a(a-1)}{\left(b-2Q^2+\sqrt{b^2+4Q^2}\right)^2+a\left(b+\sqrt{b^2+4Q^2}\right)^2+4Q^2a(a+1)},
 \end{equation}
 with
$$
    b=Q^2+a^2-1
$$ and
$$
a=\frac{1+\kappa}{1-\kappa},
$$
while the quasi-kernel $\hat A(Q)$ of $\RE\bA(Q)$ of the L-system $\Theta(Q)$ is defined by \eqref{DOMHAT} with $U(Q)$ given by the same formula \eqref{e-75-U} with the only difference that  $\kappa$ is  \eqref{e-85-kappa-prime}.

Figure \ref{fig-3} shows the  graph of $\kappa$ as a function of $Q$. Note that the vertex of the graph is located at the value of $\kappa=\kappa_0=\frac{a-1}{1+a}$. Moreover, if $a\rightarrow+\infty$, then $\kappa_0\rightarrow 1$ as indicated on the picture with the dashed lines.

%%%%%%%%%%%
\begin{figure}
  \begin{center}
  \includegraphics[width=90mm]{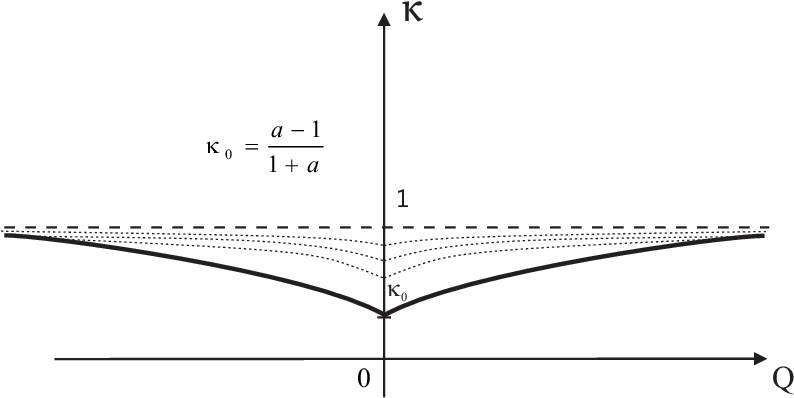}
  \caption{Class $\sM^{-1,Q}_\kappa $ ($a>1$): Parameter $\kappa$ as a function of $Q$ }\label{fig-3}
  \end{center}
\end{figure}
%%%%%%%%%%%
%\newpage

We remark that the ``perturbed" L-system $\Theta(Q)$ whose construction is based on a given L-system $\Theta$ (subject to either of Hypotheses \ref{setup} \eqref{ddoomm14} or \eqref{ddoomm14-1}) and described in details in \cite[Theorems 10.1-10.3]{BMkT-3} is called the \textbf{perturbation} of an L-system  $\Theta$. The perturbation of a given L-system  relies on the fixed choice of the deficiency vectors of the symmetric operator of $\Theta$ and a $Q$-dependent pair of von Neumann's parameters $\kappa$ and $U$ (see Appendix \ref{A1} for the exact construction).  It is important to mention that  the impedance functions of the perturbed and original L-systems are  always related  by the {\textbf{impedance shift}} formula
(cf. \eqref{impshift1}, \eqref{impshift2} and \eqref{impshift3})
$$V_{\Theta(Q)}(z)=Q+V_{\Theta}(z).$$
%\newpage

\section{c-Entropy of a perturbed L-system}\label{s5}

In this section we study how the perturbation affects the c-Entropy of an L-systems that initially satisfies the conditions of Hypotheses \ref{setup} \eqref{ddoomm14} or \eqref{ddoomm14-1}.
We begin with reminding a definition of the c-Entropy of an L-system introduced in \cite{BT-16}.
\begin{definition}
Let $\Theta$ be an L-system of the form \eqref{e6-3-2}. The quantity
\begin{equation}\label{e-80-entropy-def}
    \calS=-\ln (|W_\Theta(-i)|),%=-\ln(|\kappa|)=\ln (1/|\kappa|)
\end{equation}
where $W_\Theta(z)$ is the transfer function of $\Theta$, is called the \textbf{coupling entropy} (or \textbf{c-Entropy}) of the L-system $\Theta$.
\end{definition}
As it  mentioned  in \cite{BT-16}, there is an alternative operator-theoretic way to define the c-Entropy. If $T$ is the main operator of the L-system  $\Theta$ and $\kappa$ is  von Neumann's parameter of $T$ in some basis $g_\pm$, then,  as shown in \cite{BMkT-2}),
 $$|W_\Theta(-i)|=|\kappa|$$ and hence
\begin{equation}\label{e-70-entropy}
    \calS=-\ln (|W_\Theta(-i)|)=-\ln(|\kappa|).%=\ln (1/|\kappa|)
\end{equation}
We emphasize that c-Entropy defined by  \eqref{e-70-entropy} does not depend on the choice of deficiency basis $g_\pm$ and moreover is an additive function with respect to the coupling of L-systems (see \cite{BMkT-2}).
 Note that if, in addition,  the point $z=i$ belongs to $\rho(T)$, then we also have that
\begin{equation}\label{e-80-entropy}
     \calS=\ln (|W_\Theta(i)|)=\ln (1/|\kappa|)=-\ln(|\kappa|).
\end{equation}
This follows from the known (see \cite{ABT}) property of the transfer functions for L-systems  that states that $W_\Theta(z)\overline{W_\Theta(\bar z)}=1$ and  the fact that $|W_\Theta(i)|=1/|\kappa|$ (see \cite{BMkT}).

Now we are going to find the c-Entropy of an L-system whose impedance function belongs to the class $\sM^Q$.
\begin{theorem}\label{t-12}%
Let $\dA$ be a  symmetric densely defined closed operator  with deficiency indices $(1, 1)$ and  $(+)$-normalized deficiency vectors $g_+$ and $g_-$ and $\Theta$ be  an L-system  containing $\dA$ and satisfying Hypotheses \ref{setup} \eqref{ddoomm14} or \eqref{ddoomm14-1} with $\kappa=0$. Then for any real $Q\ne0$, the c-Entropy $\calS(Q)$ of a perturbed L-system  $\Theta(Q)$ is finite and given by the formula
\begin{equation}\label{e-45-entropy}
     \calS(Q)=\frac{1}{2}\ln (Q^2+4)-\ln|Q|.
\end{equation}
\end{theorem}
\begin{proof}
We have shown in \cite[Theorem 10.1]{BMkT-3} that if an L-system  containing $\dA$ and satisfying Hypotheses \ref{setup} \eqref{ddoomm14} or \eqref{ddoomm14-1} with $\kappa=0$ is perturbed by any real $Q\ne0$, then the parameter $\kappa(Q)$ of the perturbed L-system $\Theta(Q)$ is determined by the formula  \eqref{e-53-kappa'}. Thus, in order to find the c-Entropy of the perturbed L-system $\Theta(Q)$ we apply \eqref{e-70-entropy} to the value of $\kappa(Q)$ in \eqref{e-53-kappa'}. We get
$$
\calS(Q)=-\ln(|\kappa(Q)|)=\ln (1/|\kappa(Q)|)=\ln\frac{\sqrt{Q^2+4}}{|Q|}=\frac{1}{2}\ln (Q^2+4)-\ln|Q|,
$$
as desired \eqref{e-45-entropy}.
\end{proof}
The graph of $\calS(Q)$ as a function of $Q$ for the perturbed class $\sM^{Q}$ is shown on Figure \ref{fig-4}. We note that c-Entropy $\calS(Q)$ is infinite when $Q=0$ and  tends to zero as $Q\rightarrow\pm\infty$.

%%%%%%%%%%%
\begin{figure}
  \begin{center}
  \includegraphics[width=60mm]{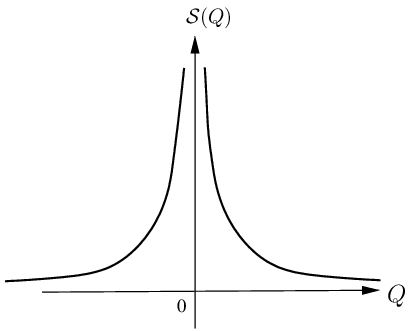}
  \caption{c-Entropy of the perturbed class $\sM^{Q}$}\label{fig-4}
  \end{center}
\end{figure}
%%%%%%%%%%%

A similar result takes place for the class $\sM_{\kappa}$.
\begin{theorem}\label{t-14}%
Let $\dA$ be a  symmetric densely defined closed operator  with deficiency indices $(1, 1)$ and  $(+)$-normalized deficiency vectors $g_+$ and $g_-$ and $\Theta$ be  an L-system  containing $\dA$ and satisfying Hypotheses \ref{setup} \eqref{ddoomm14}  with finite c-Entropy $\calS$. Then for any real $Q\ne0$, the c-Entropy $\calS(Q)$ of a perturbed L-system  $\Theta(Q)$ is finite and given by the formula
\begin{equation}\label{e-46-entropy}
     \calS(Q)=\ln\frac{\left(b-2Q^2-\sqrt{b^2+4Q^2}\right)^2+a\left(b-\sqrt{b^2+4Q^2}\right)^2+4Q^2a(a+1)}{\left(b-2Q^2-\sqrt{b^2+4Q^2}\right)^2-a\left(b-\sqrt{b^2+4Q^2}\right)^2+4Q^2a(a-1)},
 \end{equation}
  where
 \begin{equation}\label{e-47-b}
   a=\tanh\left(\frac{\calS}{2}\right)\textrm{ and  }\;b=Q^2+a^2-1.
\end{equation}
\end{theorem}
\begin{proof}
Our requirement of finite c-Entropy $\calS$ implies (via \eqref{e-70-entropy}) that $\kappa\ne0$. Also, Hypotheses \ref{setup} \eqref{ddoomm14} yields that $a=\frac{1-\kappa}{1+\kappa}$ is such that $0<a<1$. It follows from \eqref{e-70-entropy}  that  $\kappa=e^{-\calS}$ and hence
$$
a=\frac{1-\kappa}{1+\kappa}=\frac{1-e^{-\calS}}{1+e^{-\calS}}=\tanh\left(\frac{\calS}{2}\right).
$$
It was shown in \cite[Theorem 10.2]{BMkT-3} that  if an L-system  containing $\dA$ and satisfying Hypotheses \ref{setup} \eqref{ddoomm14}  with $\kappa\ne0$ is perturbed by any real $Q\ne0$, then the parameter $\kappa(Q)$ of the perturbed L-system $\Theta(Q)$ is determined by the formula  \eqref{e-53-kappa-prime} with $0<a<1$. Consequently, in order to find the c-Entropy of the perturbed L-system $\Theta(Q)$ we apply \eqref{e-70-entropy} to the value of $\kappa(Q)$ in \eqref{e-53-kappa-prime}. This clearly yields \eqref{e-46-entropy}.
\end{proof}
%%%%%%%%%%%
\begin{figure}
  \begin{center}
  \includegraphics[width=70mm]{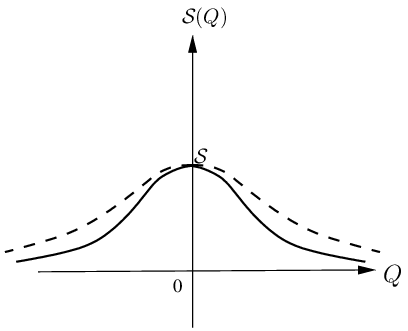}
  \caption{c-Entropy of the  classes $\sM^{Q}_\kappa$ (solid graph) and $\sM^{-1,Q}_\kappa$} (dashed graph).\label{fig-5}
  \end{center}
\end{figure}
%%%%%%%%%%%
Now we  state and prove  an analogues result for the class $\sM_{\kappa}^{-1}$.
\begin{theorem}\label{t-15}%
Let $\dA$ be a  symmetric densely defined closed operator  with deficiency indices $(1, 1)$ and  $(+)$-normalized deficiency vectors $g_+$ and $g_-$ and $\Theta$ be  an L-system  containing $\dA$ and satisfying Hypotheses \ref{setup} \eqref{ddoomm14-1} with finite c-Entropy $\calS$. Then for any real $Q\ne0$, the c-Entropy $\calS(Q)$ of a perturbed L-system  $\Theta(Q)$ is finite and given by the formula
\begin{equation}\label{e-47-entropy}
     \calS(Q)=\ln\frac{\left(b-2Q^2+\sqrt{b^2+4Q^2}\right)^2+a\left(b+\sqrt{b^2+4Q^2}\right)^2+4Q^2a(a+1)}{a\left(b+\sqrt{b^2+4Q^2}\right)^2-\left(b-2Q^2+\sqrt{b^2+4Q^2}\right)^2-4Q^2a(a-1)},
\end{equation}
  where
  \begin{equation}\label{e-48-b}
   a=\coth\left(\frac{\calS}{2}\right)\textrm{ and  }\;b=Q^2+a^2-1.
\end{equation}
 \end{theorem}
\begin{proof}
As in the proof of Theorem \ref{t-14} we note that the requirement of finite c-Entropy $\calS$ implies (via \eqref{e-70-entropy}) that $\kappa\ne0$. Also, Hypotheses \ref{setup} \eqref{ddoomm14-1} yields that $a=\frac{1+\kappa}{1-\kappa}$ is such that $a>1$. It follows from \eqref{e-70-entropy}  that  $\kappa=e^{-\calS}$ and hence
$$
a=\frac{1+\kappa}{1-\kappa}=\frac{1+e^{-\calS}}{1-e^{-\calS}}=\coth\left(\frac{\calS}{2}\right).
$$
It was shown in \cite[Theorem 10.3]{BMkT-3} that  if an L-system  containing $\dA$ and satisfying Hypotheses \ref{setup} \eqref{ddoomm14-1} with $\kappa\ne0$ is perturbed by any real $Q\ne0$, then the parameter $\kappa(Q)$ of the perturbed L-system $\Theta(Q)$ is determined by the formula  \eqref{e-85-kappa-prime} with $a>1$. Consequently, in order to find the c-Entropy of the perturbed L-system $\Theta(Q)$ we apply \eqref{e-70-entropy} to the value of $\kappa(Q)$ in \eqref{e-85-kappa-prime}. This clearly yields \eqref{e-47-entropy}.
\end{proof}

The graph of $\calS(Q)$ as a function of $Q$ for the perturbed classes $\sM^{Q}_\kappa$ (solid curve) and $\sM^{-1,Q}_\kappa$ (dashed curve) are shown on Figure \ref{fig-5}. We note that c-Entropy $\calS(Q)$ is at its maximum and equals $\calS$ when $Q=0$ and  tends to zero as $Q\rightarrow\pm\infty$.
%%%%%%%%%

\section{Dissipation coefficient of a perturbed L-system}\label{s6}

Let us recall the definition of the dissipation coefficient of an L-system.
\begin{definition}[{cf. \cite{BT-16}}, \cite{BT-21}]\label{d-10}
Let $T$ be the main operator of an L-system $\Theta$  of the form \eqref{e6-3-2} and $\kappa$ be its von {Neumann's} parameter according to a fixed  $(\cdot)$-normalized deficiency basis $g'_\pm$ such that $0\le\kappa\le1$. If \begin{equation}\label{e-76-ty}
\ti y=g'_+-\kappa g'_-,
\end{equation}
then the quantity $\calD= \IM (T \ti y,\ti y)$ is called the \textbf{coefficient of dissipation} (or dissipation coefficient) of the L-system $\Theta$.
\end{definition}
It was shown in \cite{BT-21} that the  c-entropy $\calS$ and the coefficient of dissipation $\calD$ of an L-system are related as
\begin{equation}\label{e-69-ent-dis}
\calD=1-e^{-2\cS}.
\end{equation}

We are going to find the c-Entropy of an L-system whose impedance function belongs to the class $\sM^Q$.
\begin{theorem}\label{t-16}%
Let $\dA$ be a  symmetric densely defined closed operator  with deficiency indices $(1, 1)$ and  $(+)$-normalized deficiency vectors $g_+$ and $g_-$ and $\Theta$ be  an L-system  containing $\dA$ and satisfying Hypotheses \ref{setup} \eqref{ddoomm14} or \eqref{ddoomm14-1} with $\kappa=0$. Then for any real $Q\ne0$, the dissipation coefficient $\calD(Q)$ of a perturbed L-system  $\Theta(Q)$ is  given by the formula
\begin{equation}\label{e-50-dcy}
     \calD(Q)=\frac{4}{Q^2+4}.
\end{equation}
\end{theorem}
\begin{proof}
As we did in the proof of Theorem \ref{t-12}, we use the fact that   if an L-system  containing $\dA$ and satisfying Hypotheses \ref{setup} \eqref{ddoomm14} or \eqref{ddoomm14-1} with $\kappa=0$ is perturbed by any real $Q\ne0$, then the parameter $\kappa(Q)$ of the perturbed L-system $\Theta(Q)$ is determined by the formula  \eqref{e-53-kappa'}. Consequently, in order to find the dissipation coefficient $\calD(Q)$ of the perturbed L-system $\Theta(Q)$ we apply \eqref{e-70-entropy} and \eqref{e-69-ent-dis} to the value of $\kappa(Q)$ in \eqref{e-53-kappa'}. We get
$$
\calD(Q)=1-\kappa^2(Q)=1-\frac{Q^2}{Q^2+4}=\frac{4}{Q^2+4},
$$
that confirms \eqref{e-50-dcy}.
\end{proof}

%%%%%%%%%%%
\begin{figure}
  \begin{center}
  \includegraphics[width=70mm]{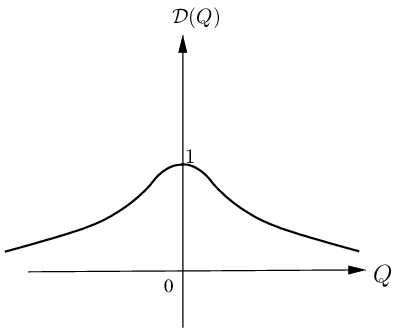}
  \caption{Dissipation coefficient of the perturbed class $\sM^{Q}$}\label{fig-6}
  \end{center}
\end{figure}
%%%%%%%%%%%

The graph of $\calD(Q)$ as a function of $Q$ for the perturbed class $\sM^{Q}$ is shown on Figure \ref{fig-6}. Note that the dissipation coefficient  $\calD(Q)$ equals $1$ when $Q=0$ and  tends to zero as $Q\rightarrow\pm\infty$.

A similar to Theorem \ref{t-16} result takes place for the class $\sM_{\kappa}$.
\begin{theorem}\label{t-17}%
Let $\dA$ be a  symmetric densely defined closed operator  with deficiency indices $(1, 1)$ and  $(+)$-normalized deficiency vectors $g_+$ and $g_-$ and $\Theta$ be  an L-system  containing $\dA$ and satisfying Hypotheses \ref{setup} \eqref{ddoomm14}   with finite c-Entropy $\calS$. Then for any real $Q\ne0$, the dissipation coefficient $\calD(Q)$ of a perturbed L-system  $\Theta_\kappa(Q)$ is  given by the formula
\begin{equation}\label{e-51-dcy}
     \calD(Q)=\frac{4(Y+Z)(X+aZ)}{(X+Y+Z(a+1))^2},
\end{equation}
  where
 \begin{equation}\label{e-52-b}
    \begin{aligned}
   a&=\tanh\left(\frac{\calS}{2}\right),\;b=Q^2+a^2-1,\; X=\left(b-2Q^2-\sqrt{b^2+4Q^2}\right)^2,\\
   Y&=a\left(b-\sqrt{b^2+4Q^2}\right)^2,\; Z=4aQ^2.
    \end{aligned}
\end{equation}
\end{theorem}
\begin{proof}
As we established  in the proof of Theorem \ref{t-14}, the requirement of finite c-Entropy $\calS$ implies (via \eqref{e-70-entropy}) that $\kappa\ne0$. Also, Hypotheses \ref{setup} \eqref{ddoomm14}  yields that $a=\frac{1-\kappa}{1+\kappa}$ is such that $0<a<1$. We have shown in the proof of Theorem \ref{t-14} that in this case
$
a=\tanh\left(\frac{\calS}{2}\right).
$
According to Section \ref{s4}, (see also  \cite[Theorem 10.2]{BMkT-3}),   if an L-system  containing $\dA$ and satisfying Hypotheses \ref{setup} with $\kappa\ne0$ is perturbed by any real $Q\ne0$, then the parameter $\kappa(Q)$ of the perturbed L-system $\Theta(Q)$ is determined by the formula  \eqref{e-53-kappa-prime} with $0<a<1$. Writing $\kappa(Q)$ from \eqref{e-53-kappa-prime} in terms of $X$, $Y$,  and $Z$ gives us
 \begin{equation}\label{e-52-kappa}
    \kappa(Q)=\frac{X-Y+(a-1)Z}{X+Y+(a+1)Z}.
 \end{equation}
  Therefore, in order to find the dissipation coefficient $\calD(Q)$ of the perturbed L-system $\Theta(Q)$ we apply \eqref{e-69-ent-dis} with \eqref{e-80-entropy-def} to the value of $\kappa(Q)$ in \eqref{e-52-kappa}.  We get, after performing some basic algebra manipulations,
$$
     \begin{aligned}
\calD(Q)&=1-\kappa^2(Q)=1-\frac{(X-Y+(a-1)Z)^2}{(X+Y+(a+1)Z)^2}=\frac{4XY+4XZ+4aZ^2+4aYZ}{(X+Y+(a+1)Z)^2}\\
&=\frac{4(Y+Z)(X+aZ)}{\left(X+Y+(a+1)Z\right)^2},
 \end{aligned}
$$
that confirms \eqref{e-51-dcy}.
\end{proof}

%%%%%%%%%%%
\begin{figure}
  \begin{center}
  \includegraphics[width=70mm]{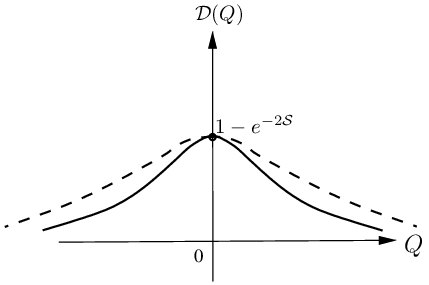}
  \caption{Dissipation coefficient of  $\sM^{Q}_\kappa$ (solid graph) and $\sM^{-1,Q}_\kappa$} (dashed graph).\label{fig-7}
  \end{center}
\end{figure}
%%%%%%%%%%%

An analogue of Theorem \ref{t-17}  for the class $\sM_{\kappa}^{-1}$ is the following.
\begin{theorem}\label{t-18}%
Let $\dA$ be a  symmetric densely defined closed operator  with deficiency indices $(1, 1)$ and  $(+)$-normalized deficiency vectors $g_+$ and $g_-$ and $\Theta$ be  an L-system  containing $\dA$ and satisfying Hypotheses \ref{setup} \eqref{ddoomm14-1}   with finite c-Entropy $\calS$. Then for any real $Q\ne0$, the dissipation coefficient $\calD(Q)$ of a perturbed L-system  $\Theta_\kappa(Q)$ is  given by the formula
\begin{equation}\label{e-54-dcy}
     \calD(Q)=\frac{4(X'+Z)(Y'+aZ)}{(X'+Y'+Z(a+1))^2},
\end{equation}
  where
 \begin{equation}\label{e-55-b}
    \begin{aligned}
   a&=\coth\left(\frac{\calS}{2}\right),\;b=Q^2+a^2-1,\; X'=\left(b-2Q^2+\sqrt{b^2+4Q^2}\right)^2,\\
   Y'&=a\left(b+\sqrt{b^2+4Q^2}\right)^2,\; Z=4aQ^2.
    \end{aligned}
\end{equation}
\end{theorem}
\begin{proof}
Following the steps of the proof of Theorem \ref{t-17}, we confirm again that the requirement of finite c-Entropy $\calS$ implies (via \eqref{e-70-entropy}) that $\kappa\ne0$. Also, Hypotheses \ref{setup}  \eqref{ddoomm14-1} yields that $a=\frac{1+\kappa}{1-\kappa}$ is such that $a>1$. We have shown in the proof of Theorem \ref{t-15} that in this case
$
a=\coth\left(\frac{\calS}{2}\right).
$
According to Section \ref{s4},    if an L-system  containing $\dA$ and satisfying Hypotheses \ref{setup}  \eqref{ddoomm14-1} with $\kappa\ne0$ is perturbed by any real $Q\ne0$, then the parameter $\kappa(Q)$ of the perturbed L-system $\Theta(Q)$ is determined by the formula  \eqref{e-85-kappa-prime} with $a>1$. Putting $\kappa(Q)$ from \eqref{e-85-kappa-prime} in a simpler form in terms of $X'$, $Y'$,  and $Z$  preset in  \eqref{e-55-b} gives us
 \begin{equation}\label{e-56-kappa}
    \kappa(Q)=\frac{Y'-X'-(a-1)Z}{X'+Y'+(a+1)Z}.
 \end{equation}
  Therefore, in order to find the dissipation coefficient $\calD(Q)$ of the perturbed L-system $\Theta(Q)$ we apply \eqref{e-69-ent-dis} with \eqref{e-80-entropy-def} to the value of $\kappa(Q)$ in \eqref{e-56-kappa}.  We get, after performing some basic algebra manipulations,
$$
     \begin{aligned}
\calD(Q)&=1-\kappa^2(Q)=1-\frac{(Y'-X'+(a-1)Z)^2}{(X'+Y'+(a+1)Z)^2}=\frac{4X'Y'+4Y'Z+4aZ^2+4aX'Z}{(X'+Y'+(a+1)Z)^2}\\
&=\frac{4(X'+Z)(Y'+aZ)}{\left(X'+Y'+(a+1)Z\right)^2},
 \end{aligned}
$$
that confirms \eqref{e-54-dcy}.
\end{proof}

\begin{table}[ht]
\centering
\begin{tabular}{|c|c|c|c|}
\hline
 &  &  &\\
 \textbf{Class}& \textbf{c-Entropy}   & \textbf{Dissipation}  & \textbf{Theorems}  \\
  &  & \textbf{coefficient} &\\ \hline%\hline
% &  &  &\\
&  &  &\\
  $\sM^Q$ & $\calS(Q)=\frac{1}{2}\ln (Q^2+4)-\ln|Q|$ & $\calD(Q)=\frac{4}{Q^2+4}$ &Theorems \ref{t-12}\\
% &  &  &\\
  &  &  & and \ref{t-16}\\  \hline
%  &  &  &\\
% &  & &\\
 &  &  &\\
 $\sM^Q_\kappa$&  Formula \eqref{e-46-entropy}& Formula \eqref{e-51-dcy} &Theorems \ref{t-14}\\
 &  &  & and \ref{t-17}\\
  % &  & &\\
 % &  &  &\\
  \hline
% &  &  &\\
%&  & &\\
 &  &  &\\
 $\sM^{-1,Q}_\kappa$  &  Formula \eqref{e-47-entropy}& Formula \eqref{e-54-dcy} &Theorems \ref{t-15}\\
  &  &  & and \ref{t-18}\\
   %&  & & \\
    % &  &  &\\
     \hline
     \multicolumn{1}{l}{} & \multicolumn{1}{l}{} & \multicolumn{1}{l}{} & \multicolumn{1}{l}{}
\end{tabular}
\caption{c-Entropy and Dissipation coefficient of perturbed L-systems}
\label{Table-1}
\end{table}

The results of Sections \ref{s5} and \ref{s6} are summarized in  Table \ref{Table-1}. We would also like to note that if $Q=0$ all the formulas for $\calS(Q)$ and $\calD(Q)$ match their ``unperturbed" versions described in \cite{BT-21}. For example, using \eqref{e-46-entropy} and \eqref{e-51-dcy} with $Q=0$ and $0<a<1$ one obtains
$$
\calS(0)=\ln(1+a)-\ln(1-a) \quad\textrm{ and }\quad \calD(0)=\frac{4a}{(1+a)^2}.
$$

%\vskip1cm

\section{Examples}

In this section we present two examples that illustrate the construction of perturbed L-system. We also   show how the c-Entropy of a perturbed L-system compares to that of an unperturbed one.

\subsection*{Example 1}\label{ex-1}

This example is designed to explain the  construction of a perturbed L-system starting with an L-system whose impedance function belongs to the class $\sM$. We will also find c-Entropy of both L-systems.

In the space $\calH=L^2_{\dR}=L^2_{(-\infty,0]}\oplus L^2_{[0,\infty)}$ we consider a prime symmetric operator
\begin{equation}\label{e-87-sym}
\dA x=i\frac{dx}{dt}
\end{equation}
on
$$
\begin{aligned}
\dom(\dA)&=\left\{x(t)=\left[
                       \begin{array}{c}
                         x_1(t) \\
                         x_2(t) \\
                       \end{array}
                     \right]\,\Big|\,x(t) -\text{abs. cont.},\right.\\
                     &\qquad\left. x'(t)\in L^2_{\dR},\, x_1(0-)=x_2(0+)=0\right\}.\\
\end{aligned}
$$
This operator $\dA$ is a model operator (see \cite{AG93}) for any prime symmetric operator with deficiency indices $(1, 1)$ that admits  dissipative extension with the point spectrum filling the entire open upper half-plane. Its  deficiency vectors  are easy to find (see \cite{AG93})
\begin{equation}\label{e-87-def}
g_z=\left(
    \begin{array}{c}
      e^{-izt}  \\
                0 \\
      \end{array}
     \right),\; \IM z>0,\qquad
g_z=\left(
\begin{array}{c}
               0 \\
      e^{-izt} \\
       \end{array}
     \right),\; \IM z<0.
\end{equation}
In particular, for $z=\pm i$ the   (normalized in $(+)$-norm) deficiency vectors  are
\begin{equation}\label{e-88-def}
g_+=\left(
    \begin{array}{c}
      e^t  \\
                0 \\
      \end{array}
     \right)
\in \sN_i,\, (t<0),\qquad
g_-=\left(
\begin{array}{c}
               0 \\
      e^{-t} \\
       \end{array}
     \right)\in \sN_{-i},\,(t>0).
\end{equation}
Consider also,
\begin{equation}\label{e-89-ext}
    \begin{aligned}
A x&=i\frac{dx}{dt},\\
\dom(A)&=\left\{x(t)=\left(
                       \begin{array}{c}
                         x_1(t) \\
                         x_2(t) \\
                       \end{array}
                     \right)
\,\Big|\,x_1(t),\,x_2(t) -\text{abs. cont.},\right.\\
&\left. x'_1(t)\in L^2_{(-\infty,0]},\, x'_2(t)\in L^2_{[0,\infty)},\,x_1(0-)=-x_2(0+)\right\}.\\
    \end{aligned}
\end{equation}
Clearly, $g_+-g_-\in\dom(A)$ and hence $A$ is  a self-adjoint extension of $\dA$ satisfying the conditions of Hypothesis \ref{setup} \eqref{ddoomm14}. Furthermore,
\begin{equation}\label{e-90-T}
    \begin{aligned}
T x&=i\frac{dx}{dt},\\
\dom(T)&=\left\{x(t)=\left(
                       \begin{array}{c}
                         x_1(t) \\
                         x_2(t) \\
                       \end{array}
                     \right)
\,\Big|\,x_1(t),\,x_2(t) -\text{abs. cont.},\right.\\
&\left. x'_1(t)\in L^2_{(-\infty,0]},\, x'_2(t)\in L^2_{[0,\infty)},\,x_2(0+)=0\right\}\\
    \end{aligned}
\end{equation}
is a quasi-self-adjoint extension of $\dA$ parameterized by a von Neumann parameter $\kappa=0$ that satisfies the conditions of Hypothesis \ref{setup} \eqref{ddoomm14}. Using direct check we obtain
\begin{equation}\label{e-91-T-star}
    \begin{aligned}
T^* x&=i\frac{dx}{dt},\\
\dom(T^*)&=\left\{x(t)=\left(
                       \begin{array}{c}
                         x_1(t) \\
                         x_2(t) \\
                       \end{array}
                     \right)
\,\Big|\,x_1(t),\,x_2(t) -\text{abs. cont.},\right.\\
&\left. x'_1(t)\in L^2_{(-\infty,0]},\, x'_2(t)\in L^2_{[0,\infty)},\,x_1(0-)=0\right\}.\\
    \end{aligned}
\end{equation}
Similarly one finds
\begin{equation}\label{e-92-adj}
    \begin{aligned}
\dot A^* x&=i\frac{dx}{dt},\\
\dom(\dot A^*)&=\left\{x(t)=\left(
                       \begin{array}{c}
                         x_1(t) \\
                         x_2(t) \\
                       \end{array}
                     \right)
\,\Big|\,x_1(t),\,x_2(t) -\text{abs. cont.},\right.\\
&\left. x'_1(t)\in L^2_{(-\infty,0]},\, x'_2(t)\in L^2_{[0,\infty)}\right\}.\\
    \end{aligned}
\end{equation}
Then $\calH_+=\dom(\dA^\ast)=W^1_2(-\infty,0]\oplus W^1_2[0,\infty)$, where $W^1_2$ is a Sobolev space.  Construct a rigged Hilbert space
\begin{equation}\label{e-139-triple}
\begin{aligned}
&\calH_+ \subset \calH \subset\calH_-\\
&=W^1_2(-\infty,0]\oplus W^1_2[0,\infty)\subset L^2_{(-\infty,0]}\oplus L^2_{[0,\infty)}\subset (W^1_2(-\infty,0]\oplus W^1_2[0,\infty))_-
\end{aligned}
\end{equation}
and consider operators
\begin{equation}\label{e-93-bA}
\begin{aligned}
\bA x&=i\frac{dx}{dt}+i x(0+)\left[\delta(t+)-\delta(t-)\right],\\
\bA^\ast x&=i\frac{dx}{dt}+i x(0-)\left[\delta(t+)-\delta(t-)\right],
\end{aligned}
\end{equation}
where $x(t)\in W^1_2(-\infty,0]\oplus W^1_2[0,\infty)$, $\delta(t+)$, $\delta(t-)$ are delta-functions and elements of $(W^1_2(-\infty,0]\oplus W^1_2[0,\infty))_-=(W^1_2(-\infty,0])_-\oplus (W^1_2[0,\infty))_-$ such that
$$
\delta(t+)=\left(
                       \begin{array}{c}
                         0 \\
                         \delta_2(t+) \\
                       \end{array}
                     \right), \qquad \delta(t-)=\left(
                       \begin{array}{c}
                        \delta_1(t-)  \\
                          0
                       \end{array}
                     \right),
$$
and generate functionals by the formulas
$$x(0+)=(x,\delta(t+))=(x_1,0)+(x_2,\delta_2(t+))=x_2(0+),$$
and
$$
x(0-)=(x,\delta(t-))=(x_1,\delta_1(t-))+(x_2,0)=x_1(0-).
$$
It is easy to see
that
$\bA\supset T\supset \dA$, $\bA^\ast\supset T^\ast\supset \dA,$
and
\begin{equation}\label{e-140-RbA}
\RE\bA x=i\frac{dx}{dt}+\frac{i }{2}(x(0+)+x(0-))\left[\delta(t+)-\delta(t-)\right].
\end{equation}
Clearly, $\RE\bA$ has its quasi-kernel equal to $A$ in \eqref{e-89-ext}. Moreover,
$$
\IM\bA =\left(\cdot,\frac{1}{\sqrt 2}[\delta(t+)-\delta(t-)]\right) \frac{1}{\sqrt 2}[\delta(t+)-\delta(t-)]=(\cdot,\chi)\chi,
$$
where $\chi=\frac{1}{\sqrt 2}[\delta(t+)-\delta(t-)]$.
Now we can build
\begin{equation}\label{e6-125-mom}
\Theta=
\begin{pmatrix}
\bA &K &1\\
&&\\
\calH_+ \subset \calH \subset\calH_- &{ } &\dC
\end{pmatrix},
\end{equation}
that is an L-system with $\calH_+ \subset \calH \subset\calH_-$ of the form \eqref{e-139-triple},
\begin{equation}\label{e7-62-new}
\begin{aligned}
Kc&=c\cdot \chi=c\cdot \frac{1}{\sqrt 2}[\delta(t+)-\delta(t-)], \quad (c\in \dC),\\
K^\ast x&=(x,\chi)=\left(x,  \frac{1}{\sqrt
2}[\delta(t+)-\delta(t-)]\right)=\frac{1}{\sqrt
2}[x(0+)-x(0-)],\\
\end{aligned}
\end{equation}
and $x(t)\in \calH_+= W^1_2(-\infty,0]\oplus W^1_2[0,\infty)$.
It was shown in \cite{BMkT-3} that   $V_{\Theta}(z)=i$ for all $z\in\dC_+$. Thus $V_{\Theta}(z)$ is a constant function of the class $\sM$. Also, clearly  the c-Entropy of the  L-system $\Theta$ in \eqref{e6-125-mom} is infinite. The corresponding dissipation coefficient of the L-system $\Theta$ is found according to \eqref{e-69-ent-dis} and is $\calD=1$.

Now let us consider
\begin{equation}\label{e-001-ex}
    V(z)=1+V_{\Theta}(z)=1+i, \quad z\in\dC_+.
\end{equation}
Obviously, by construction $V(z)\in\sM^1$. We are going to construct a perturbed L-system $\Theta(1)$ that realizes $V(z)$. This construction was thoroughly described in \cite[Example 1]{BMkT-3} and uses  the same symmetric operator $\dA$ and state-space as in L-system $\Theta$. Taking $Q=1$ in \eqref{e-53-kappa'} we  obtain
\begin{equation}\label{e-002-ex'}
\kappa(1)=\frac{1}{\sqrt5}.
\end{equation}
Then applying \eqref{e-54-U-M-q} yields
\begin{equation}\label{e-158-U}
    U(1)=\frac{-1+2i}{\sqrt5}.
\end{equation}
The L-system $\Theta(1)$ constructed in \cite{BMkT-3} with parameters $\kappa=\kappa(1)$ in \eqref{e-002-ex'} and $U=U(1)$  in \eqref{e-158-U} out of the L-system $\Theta$ is such that $V_{\Theta(1)}(z)=V(z)\equiv 1+i$, $(z\in\dC_+)$. Its main operator $T(1)$
\begin{equation}\label{e-90-T1}
    \begin{aligned}
T(1) x&=i\frac{dx}{dt},\\
\dom(T(1))&=\left\{x(t)=\left(
                       \begin{array}{c}
                         x_1(t) \\
                         x_2(t) \\
                       \end{array}
                     \right)
\,\Big|\,x_1(t),\,x_2(t) -\text{abs. cont.},\right.\\
%&\left. x'_1(t)\in L^2_{(-\infty,0]},\, x'_2(t)\in L^2_{[0,\infty)},\,\sqrt5 x_2(0+)=-x_1(0-)\right\}.\\
&\left. x'_1(t)\in L^2_{(-\infty,0]},\, x'_2(t)\in L^2_{[0,\infty)},\,5 x_2(0+)=(-1+2i)x_1(0-)\right\}.\\
    \end{aligned}
\end{equation}
and is such that $\kappa=\frac{1}{\sqrt5}$ is its von Neumann parameter of $T(1)$. The operator
\begin{equation}\label{e-161-ext}
    \begin{aligned}
A(1) x&=i\frac{dx}{dt},\\
\dom(A(1))&=\left\{x(t)=\left(
                       \begin{array}{c}
                         x_1(t) \\
                         x_2(t) \\
                       \end{array}
                     \right)
\,\Big|\,x_1(t),\,x_2(t) -\text{abs. cont.},\right.\\
%&\left. x'_1(t)\in L^2_{(-\infty,0]},\, x'_2(t)\in L^2_{[0,\infty)},\,\sqrt{5}x_2(0+)=(-1+2i)x_1(0-)\right\}.\\
&\left. x'_1(t)\in L^2_{(-\infty,0]},\, x'_2(t)\in L^2_{[0,\infty)},\,{5}x_2(0+)=(3+4i)x_1(0-)\right\}.\\
    \end{aligned}
\end{equation}
is a self-adjoint extension of $\dA$ and has the von Neumann's parameter $U$ from \eqref{e-158-U}.
Finally the state-space operator $\bA(1)$ is given by the formula
$$
\begin{aligned}
\bA(1) x&=i\frac{dx}{dt}-\frac{i}{2\sqrt5}\Big({5} x(0+)+(1-2i) x(0-)\Big)\big((5+5i)\delta(t-)+(7+i)\delta(t+)\big).
\end{aligned}
$$
The perturbed  L-system $\Theta(1)$ we desire is
\begin{equation*}%\label{e6-125}
\Theta(1)= %\left(
\begin{pmatrix}%{smallmatrix}
\bA(1)&K(1) &1\\
&&\\
\calH_+ \subset \calH \subset\calH_- &{ } &\dC
\end{pmatrix},
\end{equation*}
where $\calH_+ \subset \calH \subset\calH_-$ is of the form \eqref{e-139-triple}, $K(1) c=c\cdot \chi(1)$, $(c\in \dC)$, $K^*(1) x=(x,\chi(1))$, and $x(t)\in \calH_+$ where
\begin{equation}\label{e-160-chi1}
\chi(1)=\frac{1+i}{\sqrt{2}}\delta(t-)+\frac{7+i}{5\sqrt{2}}\delta(t+).
\end{equation}
This L-system $\Theta(1)$ has the impedance function $V_{\Theta(1)}(z)=1+i$, ($z\in\dC_+$). The c-Entropy of this perturbed L-system $\Theta(1)$ is (see \eqref{e-70-entropy})
\begin{equation}\label{e-68-entr}
\calS(1)=-\ln|\kappa|=-\ln\frac{1}{\sqrt5}=\frac{1}{2}\ln5\approx 0.8047,
\end{equation}
while the c-Entropy of the unperturbed L-system $\Theta$ is infinite. The dissipation coefficient of L-system $\Theta(1)$ is found according to \eqref{e-50-dcy}
\begin{equation}\label{e-76-dcy}
     \calD(1)=\frac{4}{1^2+4}=\frac{1}{5}.
\end{equation}

%\vskip.5cm

\subsection*{Example 2}\label{ex-2}

%\noindent\textbf{Example 2.}
In this Example  we  construct a perturbed L-system based on a given one with finite c-Entropy. We will rely on some objects presented in Example 1 but with certain changes. Consider an L-system
\begin{equation}\label{e-154-mom_0}
\Theta=
\begin{pmatrix}
\bA &K &1\\
&&\\
\calH_+ \subset \calH \subset\calH_- &{ } &\dC
\end{pmatrix}.
\end{equation}
The state space of $\Theta$ is $\calH_+ \subset \calH \subset\calH_-$ of the form \eqref{e-139-triple}  and its symmetric operator $\dA$ is given by \eqref{e-87-sym} as in Example 1.
%Let $\dA$ and $A$ be still defined by formulas \eqref{e-87-sym} and \eqref{e-89-ext}, respectively.
The main operator $T$ of $\Theta$ is defined as follows
\begin{equation}\label{e-149-T}
    \begin{aligned}
T x&=i\frac{dx}{dt},\\
\dom(T)&=\left\{x(t)=\left(
                       \begin{array}{c}
                         x_1(t) \\
                         x_2(t) \\
                       \end{array}
                     \right)
\,\Big|\,x_1(t),\,x_2(t) -\text{abs. cont.},\right.\\
&\left. x'_1(t)\in L^2_{(-\infty,0]},\, x'_2(t)\in L^2_{[0,\infty)},\,3 x_2(0+)=-x_1(0-)\right\}.\\
    \end{aligned}
\end{equation}
It follows from \eqref{e-88-def} that $g_+-\frac{1}{3}g_-\in\dom(T)$ and hence $\kappa=\frac{1}{3}$ is the von Neumann parameter of $T$ corresponding to the deficiency vectors \eqref{e-88-def}.
The state-space operator of $\Theta$ in the rigged Hilbert space \eqref{e-139-triple} is (see \cite[Example 2]{BMkT-3})
\begin{equation}\label{e-152-bA}
\begin{aligned}
\bA x&=i\frac{dx}{dt}+\frac{i}{2} (3 x(0+)+x(0-))\left[\delta(t+)-\delta(t-)\right],\\
\bA^* x&=i\frac{dx}{dt}+\frac{i}{2}  ( x(0+)+3 x(0-))\left[\delta(t+)-\delta(t-)\right],
\end{aligned}
\end{equation}
where all the components are defined in Example 1. Finally, the channel operator of L-system $\Theta$ is
\begin{equation*}\label{e-155-new_0}
\begin{aligned}
Kc&=c\cdot {\frac{1}{\sqrt2}}[\delta(t+)-\delta(t-)], \quad (c\in \dC),\\
K^* x&={\frac{1}{\sqrt2}}(x(0+)-x(0-)),\\
\end{aligned}
\end{equation*}
and $x(t)\in \calH_+= W^1_2(-\infty,0]\oplus W^1_2[0,\infty)$. It was shown in \cite[Example 2]{BMkT-3} that
$$
V_{\Theta}(z)\equiv \frac{3-1}{3+1}i=\frac{1}{2}i,\quad z\in\dC_+.
$$
Observe that $V_{\Theta}$ belongs to the class $\sM_{1/3}$, (here $\kappa=\frac{1}{3}$), and $a$ given by \eqref{e-66-L} is
$$
a=\frac{1-\kappa}{1+\kappa}=\frac{1-1/3}{1+1/3}=\frac{1}{2}.
$$
The c-Entropy of this L-system $\Theta$ is (see \eqref{e-70-entropy})
\begin{equation}\label{e-72-entr}
\calS=-\ln|\kappa|=-\ln\frac{1}{3}=\ln3\approx 1.0986,
\end{equation}
The corresponding dissipation coefficient of the L-system $\Theta$ is found according to \eqref{e-69-ent-dis} and is $\calD=\frac{8}{9}$.

\begin{table}[ht]
\centering
\begin{tabular}{|c|c|c|c|}
\hline
 &  &  &\\
 \textbf{Class}& \textbf{c-Entropy}   & \textbf{Dissipation}  & \textbf{Example}  \\
  &  & \textbf{coefficient} &\\
  \hline
  &  &  &\\
  $\sM$ & $\calS=\infty$ & $\calD=1$ &Example 1\\
% &  &  &\\
  &  &  & \\  \hline
&  &  &\\
  $\sM^1$ & $\calS(1)=\frac{1}{2}\ln5\approx 0.8047$ & $\calD(1)=\frac{1}{5}$ & Example 1\\
% &  &  &\\
  &  &  & \\  \hline
%  &  &  &\\
% &  & &\\
 &  &  &\\
 $\sM_{1/3}$&  $\calS=\ln3\approx 1.0986$& $\calD=\frac{8}{9}\approx0.8889$ &Example 2\\
 &  &  & \\
   %&  & &\\
 % &  &  &\\
  \hline
% &  &  &\\
%&  & &\\
 &  &  &\\
 $\sM_{1/3}^1$  &  $\calS(1)=\frac{1}{2}\ln\frac{13}{5}\approx 0.4778$& $\calD(1)=\frac{104}{169}\approx0.6154$ &Example 2\\
  &  &  & \\
%   &  & & \\
    % &  &  &\\
     \hline
     \multicolumn{1}{l}{} & \multicolumn{1}{l}{} & \multicolumn{1}{l}{} & \multicolumn{1}{l}{}
\end{tabular}
\caption{Numerical values of c-Entropy and Dissipation coefficient of perturbed L-systems}
\label{Table-2}
\end{table}

Now we are going to  construct a perturbed L-system $\Theta(1)$ out of the elements of L-system $\Theta(1)$ such that
$$V_{\Theta(1)}(z)= 1+\frac{1}{2}i,\quad z\in\dC_+.$$
Clearly, by construction $V_{\Theta(1)}(z)\in\sM^1_{1/3}$. As we have shown in \cite[Example 2]{BMkT-3} this construction will require the value of $\kappa(Q)$ of the form \eqref{e-53-kappa-prime} and $U(Q)$ of the form \eqref{e-75-U} for $Q=1$ to yield
\begin{equation}\label{e-161-k-U}
    \kappa(1)=\frac{\sqrt{65}}{13}\quad \textrm{and} \quad U(1)=\frac{-7+4i}{\sqrt{65}}.
\end{equation}
 Then the main operator of the constructed L-system is
\begin{equation*}\label{e-162-T1}
    \begin{aligned}
&\quad\quad T(1) x=i\frac{dx}{dt},\\
&\dom(T(1))=\left\{x(t)=\left[
                       \begin{array}{c}
                         x_1(t) \\
                         x_2(t) \\
                       \end{array}
                     \right]
\,\Big|\,x_1(t),\,x_2(t) -\text{abs. cont.}, x'_1(t)\in L^2_{(-\infty,0]},\right.\\
&\left.  x'_2(t)\in L^2_{[0,\infty)},\,\sqrt{65} x_2(0+)=-13\,x_1(0-)\right\}.\\
    \end{aligned}
\end{equation*}
Also, as it was shown in \cite[Example 2]{BMkT-3} the state-space operator of this L-system $\Theta(1)$ is
$$
\begin{aligned}
\bA(1) x&=i\frac{dx}{dt}-\frac{i}{20}\Big(\sqrt{65} x(0+)+13x(0-)\Big)\left(\sqrt{65}(4+3i) \delta(t-)+ (20+35 i)\delta(t+)\right).
\end{aligned}
$$
and the composed  L-system is
\begin{equation}\label{e6-125-11}
\Theta(1)= %\left(
\begin{pmatrix}%{smallmatrix}
\bA(1)&K(1) &1\\
&&\\
\calH_+ \subset \calH \subset\calH_- &{ } &\dC
\end{pmatrix},
%\right),
%\quad (J=-1),
\end{equation}
where $\calH_+ \subset \calH \subset\calH_-$ is of the form \eqref{e-139-triple}, $K(1) c=c\cdot \chi(1)$, $(c\in \dC)$, $K^\ast(1) x=(x,\chi(1))$, with
\begin{equation}\label{e-164-chi1}
\chi(1)=\frac{1}{2\sqrt{65}}\left(\sqrt{65}(1+2i) \delta(t-)+ (1+18 i)\delta(t+)\right),
\end{equation}
 and $x(t)\in \calH_+$. This L-system $\Theta(1)$ is such that $V_{\Theta(1)}(z)= 1+\frac{1}{2}i$ for all  $z\in\dC_+$. The c-Entropy of this perturbed L-system $\Theta(1)$ is (see \eqref{e-70-entropy})
\begin{equation}\label{e-77-entr}
\calS(1)=-\ln|\kappa(1)|=-\ln\frac{\sqrt{65}}{13}=\frac{1}{2}\ln\frac{13}{5}\approx 0.4778.
\end{equation}

Note that $\calS(1)<\calS$ where $\calS$ is given by \eqref{e-72-entr}.
The corresponding dissipation coefficient of the L-system $\Theta(1)$ is found according to \eqref{e-69-ent-dis} (or \eqref{e-51-dcy}) and is $$\calD(1)=1-\kappa^2(1)=1-\frac{65}{169}=\frac{104}{169}.$$

The numerical values for c-Entropy and Dissipation coefficient of perturbed L-systems constructed in the examples  are summarized in  Table \ref{Table-2}.

%\vskip4cm

\appendix

%\appendix
\section{Inclusion into an L-system}\label{A1}
In  this appendix, following   \cite{BMkT-3,T69},  we  provide an explicit construction of  an L-system based upon the following operator theoretic setting.

Assume that  $\dA$ is  a densely defined closed symmetric operator with finite deficiency indices $(1,1)$. Given $$
(\kappa,U)\in [0,1)\times \mathbb{T}, \quad \text{with} \quad \mathbb{T}=\{z\in \bbC\,\mid\, |z|=1\},
$$
and $(+)$-normalized deficiency elements $g_\pm\in\sN_{\pm i}=\Ker (\dA^*\mp i I)$,   $\|g_\pm\|_+=1$, assume that $T$ is a quasi-selfadjoint extension of $\dot A$ such that
$$
g_+-\kappa g_-\in \dom (T).
$$
Also assume that $A$ is a reference self-adjoint extension of $\dot A$ with
\begin{equation}\label{e-78-ex}
g_++Ug_-\in \dom (A).
\end{equation}

Introduce the  L-system  (see \cite{BMkT-3,BT-21})
 \begin{equation}\label{e-215}
\Theta= \begin{pmatrix} \bA&K&\ 1\cr \calH_+ \subset \calH \subset
\calH_-& &\dC\cr \end{pmatrix},
\end{equation}
where $\bA$ is a unique $(*)$-extension $\bA$ of $T$ (see \cite[Theorem 4.4.6]{ABT}) and
$$K\,c=c\cdot\chi,\quad  (c\in\dC).$$

In this case, the state-space  operator $\bA$ given by
\begin{equation}\label{e-205-A}
\begin{aligned}
\bA&=\dA^*+\frac{\sqrt2 i(\kappa+\bar U)}{|1+\kappa U|\sqrt{1-\kappa^2}}\Big(\cdot\,\,, \kappa\varphi+\psi\Big)\chi,
   \end{aligned}
\end{equation}
with
\begin{equation}\label{e-212}
    \chi=\frac{\kappa^2+1+2\kappa U}{\sqrt2|1+\kappa U|\sqrt{1-\kappa^2}}\varphi+ \frac{\kappa^2 U+2\kappa+ U}{\sqrt2|1+\kappa U|\sqrt{1-\kappa^2}}\psi.
\end{equation}
Here
\begin{equation}\label{e-23-phi-psi}
\varphi=\calR^{-1}(g_+),\quad  \psi=\calR^{-1}(g_-),
 \end{equation}
with  $\calR$ the   Riesz-Berezansky   operator.

\begin{remark}\label{r-1}
Notice that since by the hypothesis
$
\|g_\pm\|_+=1,
$
we have
$$\|\varphi\|_-=\|\psi\|_-=1.$$
Indeed, by  \eqref{e3-4},
$$
\|\varphi\|_-^2=\|\cR\varphi\|_+^2=\|g_+\|_+^2=1.
$$
Analogously,
$$
\|\psi\|_-^2=1.
$$
Moreover, since obviously
$$
\|g_\pm\|_+^2=2\|g_\pm\|^2,
$$
we also see that  the deficiency elements  $g_\pm'\in\sN_{\pm i}$ given by
\begin{equation}\label{e-34-conv}
    g_+'=\sqrt2\calR=\sqrt2\, g_+,\qquad g_-'=\sqrt2\calR\psi=\sqrt2\, g_-
\end{equation}
are  $(\cdot)$-normalized.
\end{remark}
{Given all that, it is also worth mentioning that  all the results are formulated in terms of  the $(+)$-normalized deficiency elements $g_\pm$.}

Observe that the constructed $L$-system  $\Theta$ of the form \eqref{e-215} is in one-to-one correspondence with a parametric pair $(\kappa,U)\in [0,1)\times \mathbb T$.
Also recall that   (see \cite{BMkT-3,BT-21})
\begin{equation}\label{e-26-Im}
        \IM\bA =(\cdot,\chi)\chi,
\end{equation}
and
\begin{equation}\label{e-214}
    \begin{aligned}
    \RE\bA&=\dA^*-\frac{i\sqrt{1-\kappa^2}}{\sqrt2|1+\kappa U|}(\cdot,\varphi-U\psi)\chi,
    \end{aligned}
\end{equation}
where $\chi$ is given by \eqref{e-212}.

If  the reference  self-adjoint extension $A$ is such that $U=-1$ in \eqref{e-78-ex},
then for the corresponding L-system
\begin{equation}\label{e-62-1-1}
\Theta_1= \begin{pmatrix} \bA_1&K_1&\ 1\cr \calH_+ \subset \calH \subset
\calH_-& &\dC\cr \end{pmatrix}
\end{equation}
  we have
  \begin{equation}\label{e-29-bA1}
    \bA_1=\dA^*-\frac{\sqrt2 i}{\sqrt{1-\kappa^2}}   \Big(\cdot, \kappa\varphi+\psi\Big)\chi_1,
     \end{equation}
where
\begin{equation}\label{e-18}
    \chi_1=\sqrt{\frac{1-\kappa}{2+2\kappa}}\,(\varphi- \psi)=\sqrt{\frac{1-\kappa}{1+\kappa}}\left(\frac{1}{\sqrt2}\,\varphi- \frac{1}{\sqrt2}\,\psi\right).
\end{equation}

  %or \eqref{e-22-AntHyp}.

  Also, \eqref{e-26-Im} gives us
\begin{equation}\label{e-17}
    \begin{aligned}
    \IM\bA_1&=\left(\frac{1}{2}\right)\frac{1-\kappa}{1+\kappa}(\cdot,\varphi-\psi)(\varphi- \psi)=(\cdot,\chi_1)\chi_1,
       \end{aligned}
\end{equation}
and, according to \eqref{e-214},
\begin{equation}\label{e-17-real}
    \begin{aligned}
    \RE\bA_1&=\dA^*-\frac{i}{2}(\cdot,\varphi+\psi)(\varphi-\psi).
       \end{aligned}
\end{equation}

If  in  \eqref{e-78-ex} we have  $U=1$, then
 the entries of the corresponding
 L-system \begin{equation}\label{e-62-1-3}
\Theta_2= \begin{pmatrix} \bA_2&K_2&\ 1\cr \calH_+ \subset \calH \subset
\calH_-& &\dC\cr \end{pmatrix}
\end{equation}
are given by
\begin{equation}\label{e-29-bA2}
    \bA_2=\dA^*+\frac{\sqrt2 i}{\sqrt{1-\kappa^2}}   \Big(\cdot, \kappa\varphi+\psi\Big)\chi_2,
     \end{equation}
where
\begin{equation}\label{e-18-1}
    \chi_2=\sqrt{\frac{1+\kappa}{2-2\kappa}}\,(\varphi+ \psi)=\sqrt{\frac{1+\kappa}{1-\kappa}}\left(\frac{1}{\sqrt2}\,\varphi+ \frac{1}{\sqrt2}\,\psi\right).
\end{equation}
Also, \eqref{e-26-Im} yields
\begin{equation}\label{e-17-1}
       \IM\bA_2= \left(\frac{1}{2}\right)\frac{1+\kappa}{1-\kappa}\Big((\cdot,\varphi+\psi)(\varphi+\psi)\Big)=(\cdot,\chi_2)\chi_2,
    \end{equation}
and, according to \eqref{e-214},
\begin{equation}\label{e-32-real}
    \begin{aligned}
    \RE\bA_2&=\dA^*-\frac{i}{2}(\cdot,\varphi-\psi)(\varphi+ \psi).
    \end{aligned}
\end{equation}

Note that two L-systems $\Theta_1$ and $\Theta_2$ in \eqref{e-62-1-1} and \eqref{e-62-1-3} are constructed in a way that the quasi-kernels $\hat A_1$ of $\RE\bA_1$ and $\hat A_2$ of $\RE\bA_2$ satisfy the conditions \eqref{ddoomm14} or \eqref{ddoomm14-1}, respectively, as it follows from \eqref{e-17-real} and \eqref{e-32-real}.

 {Also, we would like to emphasize that formulas \eqref{e-215}--\eqref{e-212} allow us to construct an  L-system $\Theta$ that is complectly  based on a given triple $(\dot A, \whA, A)$ and a fixed $(+)$-normalized deficiency vectors $g_\pm$. Moreover, in this construction the operators  $\dA$ and $T$ become the symmetric and main operators of $\Theta$, respectively, while the self-adjoint reference extension  $A$ of the triple matches $\hat A$,  the quasi-kernel of $\RE\bA$.}

%%%%%%%%%%%%%%%%%

\end{document}